\documentclass[a4paper, 11pt]{amsart}

\usepackage{amsmath, amsthm, amsfonts, amssymb, enumerate}
\usepackage[bookmarks=false]{hyperref} 
\usepackage{xcolor}
\usepackage{verbatim}
\usepackage[normalem]{ulem}
\usepackage{tikz-cd}

\title{First $\ell^2$-Betti numbers and proper proximality}
\author{Changying Ding}
\address{Department of Mathematics, Vanderbilt University, 1326 Stevenson Center, Nashville, TN 37240, USA}
\email{changying.ding@vanderbilt.edu}

\newtheorem{thm}{Theorem}[section]
\newtheorem{prop}[thm]{Proposition}

\newtheorem{cor}[thm]{Corollary}
\newtheorem{lem}[thm]{Lemma}
\theoremstyle{definition}

\newtheorem{defn/lem}[thm]{Definition/Lemma}
\newtheorem{rem}[thm]{Remark}

\newcommand{\B}{{\mathbb B}}
\newcommand{\C}{{\mathbb C}}

\newcommand{\K}{{\mathbb K}}
\newcommand{\M}{{\mathbb M}}
\newcommand{\N}{{\mathbb N}}

\newcommand{\R}{{\mathbb R}}
\newcommand{\bS}{{\mathbb S}}
\newcommand{\T}{{\mathbb T}}

\newcommand{\X}{{\mathbb X}}
\newcommand{\Z}{{\mathbb Z}}

\newcommand{\cP}{{\mathcal P}}

\newcommand{\cF}{{\mathcal F}}

\newcommand{\cH}{{\mathcal H}}

\newcommand{\cK}{{\mathcal K}}

\newcommand{\cN}{{\mathcal N}}
\newcommand{\cO}{{\mathcal O}}

\newcommand{\cU}{{\mathcal U}}

\newcommand{\cZ}{{\mathcal Z}}

\newcommand{\Ad}{\operatorname{Ad}}

\newcommand{\Aut}{\operatorname{Aut}}

\newcommand{\id}{\operatorname{id}}

\newcommand{\supp}{\operatorname{supp}}

\newcommand{\SL}{\operatorname{SL}}

\newcommand{\diag}{\operatorname{diag}}

\newcommand{\ovt}{\, \overline{\otimes}\,}

\newcommand{\ds}{{\sharp\kern-.5pt\sharp}}

\newcommand{\actson}{{\, \curvearrowright \,}}

\newcommand{\reforder}[1]{}

\makeatletter
\DeclareRobustCommand\frownotimes{\mathbin{\mathpalette\frown@otimes\relax}}
\newcommand{\frown@otimes}[2]{%
  \vbox{
    \ialign{##\cr
      \hidewidth$\m@th#1{}_\frown$\kern-\scriptspace\hidewidth\cr
      \noalign{\nointerlineskip\kern-1pt}
      $\m@th#1\otimes$\cr
    }%
  }%
}
\makeatother

\begin{document}
\maketitle
\begin{abstract}
We show that for a countable exact group, having positive first $\ell^2$-Betti number implies proper proximality
	in this sense of \cite{BoIoPe21}.
This is achieved by showing a cocycle superrigidty result for Bernoulli shifts of non-properly proximal groups.
We also obtain that Bernoulli shifts of countable, nonamenable, i.c.c., exact, non-properly proximal groups are OE-superrigid.
\end{abstract}

\section{Introduction}
The group measure space construction associates to every probability measure preserving (p.m.p.) action
	$\Gamma\actson (X,\mu)$ of a countable group $\Gamma$, a finite von Neumann algebra $M=L^\infty(X)\rtimes\Gamma$ \cite{MvN36}.
If the action is free and ergodic, then $M$ is a II$_1$ factor and $L^\infty(X)$ is a Cartan subalgebra.
During the last two decades, Popa's deformation/rigidity theory
	has led to spectacular progress in the classification
	and structural results of II$_1$ factors (see surveys \cite{Po07B, Va10, Io18}).
In particular, several large families of group measure space II$_1$ factors $L^\infty(X)\rtimes \Gamma$ 
	have been shown to have a unique Cartan subalgebra, up to unitary conjugacy 
	\cite{OzPo10I, OzPo10II, ChSi13, ChSiUd13, PoVa14I, PoVa14II, BoIoPe21}.
Such unique Cartan subalgebra results play a crucial role in the classification of group measure space II$_1$ factors,
	as they allow one to reduce the classification of the factors $L^\infty(X)\rtimes\Gamma$,   
	up to isomorphism, to the classification of the corresponding actions $\Gamma\actson X$, up to orbit equivalence
	\cite{Si55, FeMo77}. 

Partially motivated by this question, Boutonnet, Ioana and Peterson introduce the the notion of properly proximal groups in \cite{BoIoPe21},
	and they show that, among other results, $L^\infty(X)\rtimes\Gamma$ has a unique weakly compact Cartan subalgebra, up to unitary conjugacy,
	provided $\Gamma$ is properly proximal and $\Gamma\actson X$ is a free ergodic p.m.p\ action.
Properly proximal groups form a robust family, which includes 
	lattices in noncompact semisimple Lie groups, nonamenable biexact groups, nonelementary convergence groups \cite{BoIoPe21},
	CAT$(0)$ cubical groups, nonelementary mapping class groups \cite{HoHuLe20},
	wreath products $\Lambda \wr \Gamma$ with $\Lambda$ nontrivial and $\Gamma$ nonamenable \cite{DKE22},
	and is stable under measure equivalence and ${\rm W}^*$-equivalence \cite{IsPeRu19},
while as shown in \cite{IsPeRu19}, inner amenability is not the only obstruction to proper proximality.
Notably, \cite{BoIoPe21} demonstrates the first ${\rm W}^*$-strong rigidity result for $\SL_n(\Z)$ with $n\geq 3$.

Another class of groups whose associated II$_1$ factors have been extensively studied 
	is the class of groups with positive first $\ell^2$-Betti numbers \cite{Pe09, PeSi12, Ioa12a, Ioa12b, ChPe13, ChSi13, Vae13}.
For a nonamenable countable group $\Gamma$, having positive first $\ell^2$-Betti number, $\beta^{_{(2)}}_{_1}(\Gamma)>0$,
	is equivalent to the existence of unbounded cocycle into its left regular representation.
Popa and Vaes conjecture that $L^\infty(X)\rtimes \Gamma$ has a unique Cartan subalgebra, up to unitary conjugacy,
	for any free ergodic p.m.p.\ action $\Gamma\actson X$ given $\beta^{_{(2)}}_{_1}(\Gamma)>0$ (see also \cite[Probelm I]{Io18}).
In their breakthrough work \cite{PoVa14I}, 
	Popa and Vaes verify this conjecture if $\Gamma$ is, in addition, weakly amenable.

In this paper, we establish the connection between first $\ell^2$-Betti numbers and proper proximality, under a mild technical assumption.

\begin{thm}\label{thm: betti number}
Let $\Gamma$ be a countable exact group.
If $\beta_1^{(2)}(\Gamma)>0$, then $\Gamma$ is properly proximal.
\end{thm}

One concrete class of groups that satisfy the assumption of Theorem~\ref{thm: betti number} is 
	the class of one relator groups with at least $3$ generators \cite{Gue02, DiLi07},
	which was not known to be properly proximal before.

Since weak amenability implies exactness (see e.g. \cite[Proposition 2]{Kir95a} and \cite{Oza07}), 
	Theorem~\ref{thm: betti number} together with \cite[Theorem 1.5]{BoIoPe21}
	implies that for a weakly amenable group $\Gamma$ with $\beta^{_{(2)}}_{_1}(\Gamma)>0$,
	$L\Gamma$ has no Cartan subalgebras
	and $L^\infty(X)\rtimes\Gamma$ has a unique Cartan subalgebra, up unitary conjugacy, 
	for any action free ergodic p.m.p.\ $\Gamma\actson (X,\mu)$,
	which recovers the result in \cite{PoVa14I} concerning groups with positive first $\ell^2$-Betti numbers.
Although it should be noted that \cite[Theorem 1.5]{BoIoPe21} 
	follows the same general strategy as laid out in \cite{PoVa14I}.

Our approach to Theorem~\ref{thm: betti number} is rather indirect.
In fact, we first obtain the following cocycle superrigidity result for non-properly proximal groups,
	from which Theorem~\ref{thm: betti number} follows in combination with \cite[Corollary 1.2]{PeSi12}.

\begin{thm}\label{thm: cocycle superrigid}
Let $\Gamma$ be a countable group, $(X_0, \mu_0)$ be a diffuse standard probability space 
	and $\Gamma\actson (X_0^\Gamma, \mu_0^\Gamma)=:(X, \mu)$ be the Bernoulli action.
If $\Gamma$ is exact and contains a nonamenable non-properly proximal 
	wq-normal subgroup, then $\Gamma\actson (X,\mu)$ is $\{\T\}$-cocycle superrigid,
	i.e., any $1$-cocycle $w: \Gamma \times X \to \T$ is cohomologus to a homomorphism.
\end{thm}

Another theme that we explore is the rigidity of Bernoulli shifts of non-properly proximal group.
Much of the work is heavily inspired by Popa's pioneering work on Bernoulli shifts of rigid groups \cite{Po06B, Po06C, Po07}.

\begin{thm}\label{cor: conjugate}
Let $\Gamma$ be a countable group with infinite conjugacy classes (i.c.c.), 
	$(X_0, \mu_0)$ be a diffuse standard probability space 
	and $\Gamma\actson (X_0^\Gamma, \mu_0^\Gamma)=:(X, \mu)$ be the Bernoulli action.
Let $\Lambda$ be a countable group and $\Lambda \actson (Y, \nu)$ a free ergodic p.m.p.\ action such that
$L^\infty(Y)\rtimes \Lambda\cong (L^\infty(X)\rtimes\Gamma)^t$ for some $0<t\leq 1$.
If $\Lambda$ is exact and contains a nonamenable, 
	non-properly proximal normal subgroup, 
	then $t=1$ and $\Gamma\actson X$ and $\Lambda\actson Y$ are conjugate. 
\end{thm}

In particular, it follows that the fundamental group of $L^\infty(X)\rtimes\Gamma$ is trivial
	if $\Gamma$ is a countable, nonamenable, i.c.c., exact and non-properly proximal group
	and $\Gamma\actson X$ is the Bernoulli action.
Furthermore, since proper proximality and exactness are stable under measure equivalence \cite{IsPeRu19, Oza07},
	Theorem~\ref{cor: conjugate} also implies the following OE-superrigidity result.

\begin{thm}\label{cor: OE superrigidity}
Let $\Gamma$ be a countable nonamenable i.c.c.\ group $(X_0, \mu_0)$ be a diffuse standard probability space 
	and $\Gamma\actson (X_0^\Gamma, \mu_0^\Gamma)=:(X, \mu)$ be the Bernoulli action.
If $\Gamma$ is exact and non-properly proximal, then $\Gamma\actson X$ is OE-superrigid, 
	i.e., if a free ergodic p.m.p.\ action $\Lambda\actson (Y,\nu)$ is orbit equivalent to $\Gamma\actson (X,\mu)$,
	then these two actions are conjugate.
\end{thm}

As a consequence, every countable nonamenable i.c.c.\ exact group has at least one desirable rigidity property, depending on whether or not it is properly proximal:
	either every group measure space II$_1$ factor 
	has at most one weakly compact Cartan subalgebra,
	or else Bernoulli shifts are OE-superrigid.

All the above theorems are derived from the following von Neumann algebraic statement.
\begin{thm}\label{thm: absorption}
Let $\Gamma$ be a countable group, $(X_0, \mu_0)$ be a diffuse standard probability space 
	and $\Gamma\actson (X_0^\Gamma, \mu_0^\Gamma)=:(X, \mu)$ be the Bernoulli action.
Suppose $\Gamma$ is exact and $N\subset M:=L^\infty(X)\rtimes \Gamma$ is a von Neumann subalgebra that has no amenable and no properly proximal direct summand.
Then there exists a s-malleable deformation $\{\alpha_t\}_{t\in\R}$ on $M$ such that $\alpha_t\to \id_N$ uniformly 
	on the unit ball of $N$, as $t\to 0$.
\end{thm}

Here, the s-malleable deformation is in the sense of Popa \cite{Po06A, Po06B} and 
	this specific deformation is the one associated with Gaussian actions \cite{Fu07} (see Section~\ref{section: 3rd step} for details).
Proper proximality is for von Neumann algebras, in the sense of \cite{DKEP22}.
We note that Theorem~\ref{thm: cocycle superrigid} follows from Theorem~\ref{thm: absorption} together with 
	Popa's seminal work on cocycle superrigidity \cite{Po07, Po08},
	and Theorem~\ref{cor: conjugate} is a combination of Theorem~\ref{thm: absorption}
	and with Popa's conjugacy criterion for Bernoulli actions \cite{Po06C}.
Exactness of groups is crucial to our proof 
	as we exploit the fact that $\Z\wr \Gamma$ is biexact relative to $\Gamma$,
	provided $\Gamma$ is exact.

Let us finish with some comparisons between our results and some existing results on inner-amenable groups.
The family of exact, non-properly proximal groups is strictly larger than the family of exact, inner-amenable groups,
	due to \cite{IsPeRu19}, \cite{DTW20} and \cite{GHW05}, as well as the permanence properties of exactness of groups (see e.g. \cite[Section 5.1]{BrOz08}).
Moreover, exactness and proper proximality are both stable 
	under measure equivalence and ${\rm W}^*$-equivalence \cite{Oza07, IsPeRu19},
	while inner-amenability is not preserved under measure equivalence \cite{DTW20} 
	and is not known to be stable under ${\rm W}^*$-equivalence.
Therefore, Theorem~\ref{thm: cocycle superrigid} can be seen as a generalization of \cite[Theorem 11]{TD20}
	in the case of $\T$-valued cocycles associated with Bernoulli shifts of exact groups.
And under the mild assumption on exactness, 
	Theorem~\ref{thm: betti number} generalizes \cite[Corollary F]{Dri22}
	and Theorem~\ref{thm: absorption} extends \cite[Theorem E]{Dri22} in the case of wreath products.

\textbf{Comments on the proofs.}
Let us outline the proof of Theorem~\ref{thm: absorption},
	which uses the recently developed notion of proper proximality in \cite{BoIoPe21, IsPeRu19, DKEP22} and Popa's deformation/rigidity theory.
The proof is divided into three steps.
First we observe in Proposition~\ref{prop:dichotomy} that for any von Neumann subalgebra $N$ in $L(\Z\wr \Gamma)$,
	with $\Gamma$ exact, 
	if $N$ has no amenable direct summand, then it must be
	properly proximal relative to $L\Gamma$ in the sense of \cite{DKEP22}.
This is a direct adaptation of \cite[Theorem 7.1]{DKEP22},
	since $\Z\wr \Gamma$ is biexact relative to $\Gamma$ \cite[Proposition 15.3.6]{BrOz08}.
Next in Section~\ref{sec: bootstrap}, 
	we use techniques from \cite{DKE22}, which extends the idea in \cite[Lemma 3.3]{DKE21} to the von Neumann algebra setting.
Continuing in the above setting with $N$ properly proximal relative to $L\Gamma$,
	we show that $N$ either has a properly proximal direct summand 
	or is amenable relative to $L\Gamma$ inside $L(\Z\wr \Gamma)$.
In this step, the notion of normal bidual developed in \cite[Section 2]{DKEP22} is extensively used.
Lastly, using a technique from \cite{Ioa15}, 
	we conclude in Section~\ref{section: 3rd step} that if $N\subset L(\Z\wr \Gamma)$ is amenable relative to $L\Gamma$, 
	then $N$ must be rigid with respect to the s-malleable deformation $\{\alpha_t\}$ associated with $L(\Z\wr \Gamma)$.
Altogether, we obtain that if $N\subset L(\Z\wr \Gamma)$ has no amenable or properly proximal direct summand, then $N$ must be $\alpha_t$-rigid.

\textbf{Acknowledgements.}
This work was started while the author was visiting the University of California, San Diego.
The author is very grateful for the kind hospitality and stimulating environment at UCSD
and would like to thank Adrian Ioana for helpful discussions and comments on a preliminary version of this paper.
The author would like to thank Jesse Peterson for 
	 encouragement and stimulating conversations.

\section{Preliminaries}
\subsection {Popa's intertwining-by-bimodules}

\begin{thm}[\cite{Po06B}] Let $(M,\tau)$ be a tracial von Neumann algebra and $P\subset pM p,Q\subset M$ be von Neumann subalgebras. 
Then the following  are equivalent:

\begin{enumerate}
\item There exist projections $p_0\in P, q_0\in Q$, a $*$-homomorphism $\theta:p_0P p_0\rightarrow q_0Q q_0$  and a non-zero partial isometry $v\in q_0M p_0$ such that $\theta(x)v=vx$, for all $x\in p_0P p_0$.

\item There is no sequence $u_n\in\mathcal U(P)$ satisfying $\|E_Q(x^*u_ny)\|_2\rightarrow 0$, for all $x,y\in pM$.
\end{enumerate}

\end{thm}

If one of these equivalent conditions holds,  we write $P\prec_{M}Q$.

\subsection{Relative amenability}

Let $P\subset M $ and $Q \subseteq M$ be a von Neumann subalgebras.
Following \cite{OzPo10I},
	we say that $P$ is  amenable relative to $Q$ inside $M$ 
	if there exists a sequence $\xi_n\in L^2(\langle M,e_Q\rangle)$ such that $\langle x\xi_n,\xi_n\rangle\rightarrow\tau(x)$, 
	for every $x\in M $, and $\|y\xi_n-\xi_ny\|_2\rightarrow 0$, for every $y\in P$,  
	or equivalently if there exists a $P$-central state on $\langle M,e_Q\rangle$ 
	that is normal when restricted to $M$ and faithful on $\mathcal{Z}(P'\cap M)$. 

\subsection{Mixing subalgebras of finite von Neumann algebras.}

Let $M$ be a finite von Neumann algebra and $N\subset M$ a von Neumann subalgebra. Recall the inclusion $N\subset M$ is mixing if $L^2(M\ominus N)$ is mixing as an $N$-$N$ bimodule, i.e., for any sequence $u_n\in\mathcal U(N)$ converging to $0$ weakly, one has $\|E_N(xu_ny)\|_2\to 0$ for any $x,y\in M\ominus N$.
When $M$ and $N$ are both diffuse, we may replace sequence of unitaries with any uniformly bounded sequence in $N$ converging to $0$ weakly
	by the proof of (4) $\implies$ (1) in \cite[Theorem 5.9]{DKEP22}.

Examples of mixing subalgebras include $L\Lambda\subset L\Gamma$, where $\Lambda<\Gamma$ is almost malnormal
	i.e., $|t\Lambda t^{-1} \cap \Lambda|<\infty$ for any $t\in \Gamma\setminus \Lambda$ 
	(see e.g. \cite[Appendix A]{Bou14}).
\subsection{Proper proximality}
We recall the notion of properly proximal von Neumann algebras from \cite{DKEP22}.

\subsubsection{Boundary pieces}\label{sec:boundary piece}
Given $M$ a finite von Neumann algebra,
an $M$-boundary piece $\X$ is a hereditary ${\rm C}^*$-subalgebra of $\B(L^2M)$ 
	such that $\M(\X)\cap M\subset M$ and $\M(\X)\cap JMJ\subset JMJ$ are weakly dense,
	where $\M(\X)$ is the multiplier of $\X$.
To avoid pathological examples, we will always assume that $\X\neq \{0\}$
	and it follows that $\K(L^2M)\subset \X$ for any $M$-boundary piece $\X$.

Let $M$ be a finite von Neumann algebra and $\X$ an $M$-boundary piece. 
Denote by $\K^L_\X(M)\subset \B(L^2M)$ the $\|\cdot\|_{\infty, 2}$-closure of the norm closed left ideal $\B(L^2M)\X$,
	where $\|T\|_{\infty, 2}=\sup_{a\in (M)_1} \|T \hat a\|_2$ for any $T\in \B(L^2M)$,
and set $\K_\X(M)=(\K^L_\X(M))^* \cap \K^L_\X(M)$ to be the hereditary ${\rm C}^*$-subalgebra generated by $\K^L_\X(M)$.
The multiplier algebra of $\K_\X(M)$ contains both $M$ and $JMJ$ and
	we denote by $\K_\X^{\infty,1} (M)$ the $\|\cdot \|_{\infty,1}$-closure of $\K_\X(M)$, 
	where $\|T\|_{\infty,1}=\sup_{a,b\in (M)_1} \langle T\hat a, \hat b\rangle$ for $T\in \B(L^2M)$,
	and $\K_\X^{\infty,1} (M)$ coincides with $\overline{\X}^{\|\cdot\|_{\infty,1}}$.
And denote by $\bS_\X(M)$ the following operator system that contains $M$,
	$$\bS_\X(M)=\{ T\in \B(L^2M) \mid [T,x]\in \K_\X^{\infty,1}(M),\ {\rm for\ any\ }x\in JMJ \}.$$
When $\X=\K(L^2M)$, we omit $\X$ in the above notations for simplicity.

Let $N\subset M$ be a von Neuman subalgebra.
We say $N\subset M$ is properly proximal relative to $\X$ if there does not exist any $N$-central state $\varphi$ on $\bS_\X(M)$ 
	such that $\varphi_{\mid M}$ is normal.
And we say $M$ is properly proximal if $M\subset M$ properly proximal relative to $\K(L^2M)$.
By \cite[Theorem 6.2]{DKEP22}, a group $\Gamma$ is properly proximal in the sense of \cite{BoIoPe21} 
	if and only if $L\Gamma$ is properly proximal.

One particular type of boundary pieces arise from subalgebras.
Let $N\subset M$ be a von Neumann subalgebra and we may associate with $N$ 
	an $M$-boundary piece $\X_N$, which is the hereditary ${\rm C}^*$-subalgebra of $\B(L^2M)$
	generated by $x JyJ e_N$ for $x, y\in M$, where $e_N\in \B(L^2M)$ is the orthogonal projection	
	from $L^2M$ onto $L^2N$.

\begin{rem}\label{rem: no direct summand}
Let $\Gamma$ be a group that is not properly proximal, then $L\Gamma$ has no properly proximal direct summand. 
Indeed, suppose $z\in \cZ(L\Gamma)$ in a nonzero central projection such that $zL\Gamma$ is not properly proximal,
	i.e., there exists a $zL\Gamma$-central state $\varphi:\bS(z L\Gamma)\to \C$ that is normal on $zL\Gamma$.
We may consider the $\Gamma$-equivariant embedding $i:\bS(\Gamma) \to \bS(L\Gamma)$ and $E:=\Ad(z): \bS(L\Gamma)\to \bS(z L\Gamma)$ \cite[Section 6]{DKEP22}.
Then $\varphi\circ E\circ i: \bS(\Gamma)\to \C$ is then a $\Gamma$-invariant state, showing $\Gamma$ is not properly proximal.
A similar argument shows that if $\Gamma$ is nonamenable, then $L\Gamma$ has no amenable direct summand.
\end{rem}

Recall that a group $\Gamma$ is biexact relative to a subgroup $\Lambda<\Gamma$ if the left action of $\Gamma$
	on $\bS_{\Lambda}(\Gamma)=\{f\in \ell^\infty\Gamma\mid f- R_t f\in c_0(\Gamma, \{\Lambda\})\}$ is topologically amenable,
	where $c_0(\Gamma, \{\Lambda\})$ is functions on $\Gamma$ that converge to $0$ when $t\in \Gamma$ escapes subsets of $\Gamma$
	that are small relative to $\Lambda$ (See \cite[Chapter 15]{BrOz08} for the precise definition).
We remark that this is equivalent to $\Gamma\actson\bS_I(\Gamma)$ is amenable.
Indeed, since we may embed $\ell^\infty\Gamma \hookrightarrow I^{**}$ in a $\Gamma$-equivariant way, 
	we have $\Gamma\actson I^{**} \oplus (\bS_I(\Gamma)/I)^{**} = \bS_I(\Gamma)^{**}$ is amenable,
	and it follows that $\Gamma\actson \bS_I(\Gamma)$ is an amenable action \cite[Proposition 2.7]{BEW19}.

The following is an easy adaptation of \cite[Theorem 7.1]{DKEP22}.
For completeness, we include the proof.

\begin{prop}\label{prop:dichotomy}
Suppose a group $\Gamma$ is biexact relative to a subgroup $\Lambda<\Gamma$. 
	Then for every von Neumann subalgebra $N \subset L\Gamma$, 
	either the inclusion $N \subset L\Gamma$ is properly proximal relative to $\X_{L\Lambda}$ 
	or else $N$ has an amenable direct summand.
\end{prop}

\begin{proof}
Suppose the inclusion $N \subset L\Gamma$ is not properly proximal relative to $\X_{L\Lambda}$, 
	and let $\phi: \bS_{\X_{L\Lambda}}(L \Gamma) \to \langle p( L\Gamma )p, e_{Np} \rangle$ 
	be a $p(L\Gamma)p$-bimodular u.c.p.\  map, where $p \in \mathcal Z(N)$ is a non-zero central projection.

If we consider the $\Gamma$-equivariant diagonal embedding $\ell^\infty \Gamma \subset \B(\ell^2\Gamma)$, 
	we see that $c_0(\Gamma, \{\Lambda\})$ is mapped to $\X_{L\Lambda}$. 
Restricting to $\bS_\Lambda(\Gamma)$ then gives a $\Gamma$-equivariant embedding into $\bS_{\X_{L\Lambda}}(L \Gamma)$. 
We therefore obtain a $*$-homomorphism $\bS_\Lambda(\Gamma) \rtimes_f \Gamma \to \B(\ell^2 \Gamma )$ 
	whose image is contained in $\bS_{\X_{L\Lambda}}(L \Gamma)$. 
Composing this $*$-homomorphism with the u.c.p.\ map $\phi$ then gives a u.c.p.\ map $\tilde \phi: \bS_\Lambda(\Gamma) \rtimes_f \Gamma \to \langle p(L\Gamma)p, e_{Np} \rangle$ such that $\tilde \phi(t) = pu_tp$ for all $t \in \Gamma$. 

Since $\Gamma$ is biexact relative to $\Lambda$, 
	the action $\Gamma \actson \bS_\Lambda(\Gamma)$ is topologically amenable.
Hence, $ \bS_\Lambda(\Gamma) \rtimes_f \Gamma =  \bS_\Lambda(\Gamma)\rtimes_r \Gamma$ is a nuclear C$^*$-algebra.
We set $\varphi(\cdot):=\frac{1}{\tau(p)} \langle \tilde \phi(\cdot) \hat p, \hat p\rangle$ and note that for $x \in C^*_r \Gamma$ 
	we have $\varphi(x) = \frac{1}{\tau(p)} \tau(pxp)$. 
Since $C^*_r\Gamma$ is weakly dense in $L \Gamma$ 
	an argument similar to Proposition 3.1 in \cite{BoCa15} 
	then gives a representation $\pi_\varphi:\bS_\Lambda(\Gamma) \rtimes_r \Gamma\to \B(\mathcal H_\varphi)$, 
	a state $\tilde \varphi\in \B(\mathcal H_\varphi)_*$ with $\varphi = \tilde \varphi \circ \pi_\varphi$, 
	and a projection $q\in  \pi_\varphi(\bS_\Lambda(\Gamma) \rtimes_r \Gamma)''$ 
	with $\tilde \varphi(q) = 1$ such that 
	there is a normal unital $*$-homomorphism $i:L\Gamma\hookrightarrow q \pi_\varphi(\bS_\Lambda(\Gamma) \rtimes_r \Gamma)''q$.

Since $\bS_\Lambda(\Gamma) \rtimes_r \Gamma$ is nuclear, 
	we have that $\pi_\varphi(\bS_\Lambda(\Gamma) \rtimes_r \Gamma)''$ is injective, 
	and so there is a u.c.p.\ map $\tilde i: \B(\ell^2\Gamma)\to q \pi_\varphi(\bS_\Lambda(\Gamma) \rtimes_r \Gamma)''q$ 
	that extends $i$.
Notice that $\psi:=\tilde \varphi \circ \tilde i$ is then an $Np$-central state on $\B(\ell^2\Gamma)$ and $\psi(x) = \frac{1}{\tau(p)} \tau(pxp)$ for $x\in L\Gamma$. 
Therefore, $Np$ is amenable.
\end{proof}

\subsubsection{A bidual characterization}\label{sec:bidual}
Next we collect some basics of the normal bidual from \cite[Section 2]{DKEP22}.

Given a finite von Neumann algebra $M$ and a C$^*$-subalgebra $A\subset \B(L^2M)$ 
	such that $M$ and $JMJ$ are contained in the multiplier algebra $\M(A)$,
we recall that $A^{M \sharp M}$ (resp. $A^{JMJ\sharp JMJ})$ denotes the space of $\varphi \in A^*$ 
		such that for each $T \in A$ the map $M \times M \ni (a, b) \mapsto \varphi(aTb)$ 
		(resp. $JMJ\times JMJ \ni (a,b)\mapsto \varphi(aTb)$) is separately normal in each variable.
When there is no confusion about the von Neumann algebra that we are referring to, 
	we will denote $A^{M \sharp M}$  by $A^\sharp$ and $A^{M \sharp M} \cap A^{JMJ \sharp JMJ}$ by $A^\sharp_J$.

We may view $(A^\sharp_J)^*$ as a von Neumann algebra as follows.
Denote by $p_{\rm nor}\in \M(A)^{**}$ the supremum of support projections of states in $\M(A)^*$ 
	that restrict to normal states on $M$ and $JMJ$, 
	so that $M$ and $JMJ$ may be viewed as unital von Neumann subalgebras of $p_{\rm nor} \M(A)^{**} p_{\rm nor}$,
	which is canonically identified with $(\M(A)^\sharp_J)^*$.
Let $q_A\in\mathcal P(\M(A)^{**})$ be the central projection such that $q_A(\M(A)^{**})=A^{**}$ 
	and we may then identify $(A^\sharp_J)^*$ with $p_{\rm nor} q_A \M(A)^{**} p_{\rm nor}=p_{\rm nor} A^{**} p_{\rm nor}$.
Furthermore, if $B\subset A$ is another C$^*$-subalgebra with $M$, $JMJ\subset \M(B)$, 
	we may identify $(B^\sharp_J)^*$ with $q_B p_{\rm nor} A^{**} p_{\rm nor} q_B$,
	which is a non-unital subalgebra of $(A^\sharp_J)^*$.

If we denote by $\iota : \M( A)\to \M(A)^{**}$ the canonical embedding, 
	we may then view $\M(A)$ as an operator subsystem of $(\M(A)^\sharp_J)^*$ 
	through the isometric u.c.p.\ map $\iota_{\rm nor}: \M(A)\ni T\to p_{\rm nor}\iota (T) p_{\rm nor}\in (\M(A)^\sharp_J)^*$,
	and the its restriction to $A$ gives a natural embedding of $A\subset (A^\sharp_J)^*$ as an operator system.

It is worth noting that $\iota_{\rm nor}$ and $\iota$ are different in a few ways.
On one hand, $\iota: \M(A)\to \M(A)^{**}$ is a $*$-homomorphism 
	while $\iota_{\rm nor}: \M(A)\to (\M(A)^\sharp_J)^*$ is a u.c.p.\ map;
on the other hand, $\iota_{\rm nor}$ gives rise to normal faithful representations when restricted to $M$ and $JMJ$,
	but $\iota_{\mid M}$ and $\iota_{\mid JMJ}$ are not normal in general.

The following is a bidual characterization of properly proximal.

\begin{lem}{\cite[Lemma 8.5]{DKEP22}}\label{lem:bidual character}
Let $M$ be a separable tracial von Neumann algebra with an $M$-boundary piece $\mathbb X$.
Then $M$ is properly proximal relative to $\X$ if and only if there is no $M$-central state $\varphi$ on
\[
\widetilde{\bS}_{\X}(M) := \left\{ T \in \left(\B(L^2 M)_J^{{\sharp}} \right)^* \mid [T, a] \in \left( \K_\X(M)_J^\sharp \right)^* \ {\rm for \ all \ } a \in JMJ \right\}
\] 
such that $\varphi_{| M }$ is normal.
\end{lem}

When $\X=\K(L^2M)$, we will abbreviate $\X$ for simplicity.
It is worth noting that $\widetilde{\bS}_\X(M)$ is a von Neumann algebra which contains $M$ as a von Neumann subalgebra,
	while $\bS_\X(M)$ is only an operator system.
Following the above discussion, we note that $\tilde{\bS}_{\X}(M)$ may be identified with
$$\widetilde{\bS}_\X(M)= \{ T\in p_{\rm nor}\B(L^2M)^{**}p_{\rm nor}\ | \ [T,a]\in q_{\X} \big(\M(\K_\X(M))\big)^{**} q _\X , {\rm\ for\ any}\ a\in JMJ\},$$
where $q _{\X}$ is the identity of $(\K_\X( M )^\sharp_J)^* \subset (\M(\K_\X(M))^\sharp_J)^*$.
If we set $q_\K=q_{\K(L^2M)}$ to be the identity of $(\K(L^2M)^\sharp _J)^*\subset (\B(L^2M)^\sharp_J)^*$,
	then using the above description of $\tilde \bS_\X(M)$, 
	we have  $q_{\X}^{\perp} \widetilde{\bS}_{\X}(M)q_{\X}^{\perp}\subset q_\K^\perp \widetilde{\bS}(M)$,
	as $q_{\X}$ commutes with $M$ and $JMJ$.

\begin{lem}\label{lem: faithful state}
Let $M$ be a separable tracial von Neumann algebra. 
Suppose $M$ has no properly proximal direct summand,
	then there exists an $M$-central state $\varphi$ on $\widetilde{\bS}(M)$ such that $\varphi_{\mid M}$ is faithful and normal.
\end{lem}
\begin{proof}
First we show that there exists an $M$-central state $\varphi$ on $\widetilde{\bS}(M)$
such that $\varphi_{\mid \cZ(M)}$ is faithful.
Consider a pair $(\varphi, p)$, 
	where $\varphi\in \widetilde{\bS}(M)^*$ is an $M$-central state 
	such that $\varphi_{\mid M}$ is normal and $p\in \cZ(M)$ is support projection of $\varphi$. 
And we may order such pairs by the order on $\cZ(M)$, i.e., $(\varphi_1, p_1)\leq (\varphi_2, p_2)$ if $p_1\leq p_2$.
If $\{(\varphi_i, p_i)\}_{i\in I}$ is a chain, then we may find a subsequence $p_{i(n)}$ such that $\lim_n p_{i(n)}=\vee_{i\in I} p_i$, 
	and $\varphi_0(\cdot)=\sum_{n\geq 1} 2^{-n} \varphi_{i(n)}(p_{i(n)}\cdot p_{i(n)})$ 
	then is an $M$-central state on $\widetilde{\bS}(M)$ such that ${\varphi_0}_{\mid M}$ is normal and $\vee_{i\in I}p_i$ is the support of $\varphi_0$.
Suppose $(\varphi, p)$ is a maximal element and $q=p^\perp >0$.
Denote by $E_q: \Ad(q):\B(L^2M)\to \B(L^2 (qM))$ and 
	one checks that $(E_q)^*$ maps $\B(L^2 (qM))^\sharp_J$ to $\B(L^2M)^\sharp_J$.
Therefore dualizing $E_q$ yeilds a u.c.p.\ map 
	$\tilde E_q: (\B(L^2M)^\sharp_J)^* \to (\B(L^2(qM))^\sharp_J)^*$,
	and $(\tilde E_q)_{\mid \widetilde{\bS}(M)}: \widetilde{\bS}(M)\to \widetilde{\bS}(qM)$.
Since $qM$ is not properly proximal, there exists a state $\psi\in \widetilde{\bS}(qM)^*$ that is $qM$-central and $\psi_{\mid qM}$ is normal.
Set $\varphi'(T)=\varphi(pTp)+\psi(\tilde E_q(qTq))$, which is an $M$-central state on $\widetilde{\bS}(M)$ that is normal on $M$ with support strictly larger than $p$,
	which is a contradiction.

Now suppose $\varphi$ is such a state with $\varphi_{\mid \cZ(M)}$ faithful,
	and $\varphi(p)=0$ for some $p\in \cP(M)$, then we may write the central support $z(p)=\sum_{i=1}^\infty v_i v_i^*$, where
	$v_i\in M$ are partial isometries such that $v_i ^* v_i\leq p$.
Since $\varphi$ is normal and tracial on $M$, we have 
	$\varphi(z(p))=\sum_{i=1}^\infty \varphi(v_i v_i^*)\leq \sum_{i=1}^\infty \varphi(p)=0$,
	which shows that $p\leq z(p)=0$.
\end{proof}

\section{From non-proper proximality to relative amenability}\label{sec: bootstrap}

In this section, we connect non-proper proximality with relative amenability using the following result.

\begin{prop}\label{prop:bootstrap}
Let $\Gamma$ be a nonamenable countable group and $\Lambda<\Gamma$ an infinite almost malnormal subgroup.
Let $M=L\Gamma$, $\X=\X_{L\Lambda}$ the $M$-boundary piece associated with $L\Lambda$, and $N\subset M$ a von Neumann subalgebra.
Suppose $N\subset M$ is properly proximal relative to $\X$.
If $N$ does not have any properly proximal direct summand, then
$N$ is amenable relative to $L\Lambda$ inside $M$.
\end{prop}

The above proposition,
	which grew out of discussions with Srivatsav Kunnawalkam Elayavalli.
A more general version appears in \cite{DKE22} and we only present the form that is sufficient for our purpose.
Before proceeding to the proof, we collect a few auxiliary lemmas.

\subsection{Boundary pieces in the bidual}\label{section: bidual boundary piece}
Let $M$ be a finite von Neumann algebra, $\X$ an $M$-boundary piece,
	$N\subset M$ a von Neumann subalgebra and $E:= \Ad(e_N): \B(L^2M)\to \B(L^2N)$.
Notice that $(E^*)_{\mid \B(L^2M)^\sharp_J}: \B(L^2M)^\sharp_J \to \B(L^2 N)_J ^\sharp$ 
	since $E$ is a normal conditional expectation when restricted to $M$ and $JMJ$,
	and a state $\varphi\in \B(L^2M)^*$ lies in $\B(L^2M)^\sharp_J$ if and only if $\varphi_{\mid M}$ and $\varphi_{\mid JMJ}$ are normal.
Thus we may consider the u.c.p.\ map
$$\tilde E:= (E^*_{\mid \B(L^2M)^\sharp_J})^*: \left(\B(L^2 M)_J^{{\sharp}} \right)^* \to \left(\B(L^2 N)_J^{{\sharp}} \right)^*.$$

\begin{lem}\label{lem: bidual conditional expectation}
Using the above notations, we have $\tilde E_{\mid \widetilde{\bS}(M)}: \widetilde{\bS}(M)\to \widetilde{\bS}(N)$.
\end{lem}
\begin{proof}
First observe that $\tilde E$ is weak$^*$ continuous and $E:\K(L^2 M) \to \K(L^2N)$.
It follows that $\tilde E$ maps $( \K(L^2 M)_J^\sharp )^*$ to $( \K(L^2N)_J^\sharp )^*$.
Furthermore, since $\tilde E_{\mid JNJ}= E_{\mid JNJ}= \id_{JNJ}$, we have $\tilde E([T,x])=[\tilde E(T), x]$ for any $x\in JNJ$ and $T\in (\B(L^2 M)_J^{{\sharp}} )^*$.
The statement follows from the definition of $\widetilde{\bS}(M)$.
\end{proof}

Recall from Section~\ref{sec:bidual} that $\iota: \K_\X(M)\to \K_\X(M)^{**}$ is the canonical embedding, 
	$p_{\rm nor}\in \B(L^2M)^{**}$ is the projection such that $p_{\rm nor}\K_\X(M)^{**} p_{\rm nor}=(\K_\X(M)^\sharp_J)^*$
	and the embedding $\iota_{\rm nor}: \K_\X(M)\to (\K_\X(M)^\sharp_J)^*$ is given by $\iota_{\rm nor}=\Ad( p_{\rm nor})\circ \iota$.

\begin{lem}\label{lem: approx unit}
Let $M$ be a finite von Neumann algebra and $\X$ an $M$-boundary piece.
Let $\X_0\subset \K_\X(M)$ be a ${\rm C}^*$-subalgebra and $\{e_n\}_{n\in I}$ an approximate unit of $\X_0$.
If $\X_0\subset \K_\X^{\infty,1}(M)$ is dense in $\|\cdot\|_{\infty,1}$ 
	and $\iota(e_n)$ commutes with $p_{\rm nor}$ for each $n\in I$,
	then $\lim_n \iota_{\rm nor}( e_n)\in (\K_\X(M)^\sharp_J)^*$ is the identity, where the limit is in the weak$^*$ topology.
\end{lem}
\begin{proof}
Since $\iota_{\rm nor}(\K_\X(M))\subset (\K_\X(M)^\sharp_J)^*$ is weak$^*$ dense and functionals in $\K_\X(M)^\sharp_J$ are continuous in $\|\cdot\|_{\infty,1}$ topology by \cite[Proposition 3.1]{DKEP22}, 
	we have $\iota_{\rm nor}(\X_0)\subset (\K_\X(M)^\sharp_J)^*$ is also weak$^*$ dense.
Let $e=\lim_n \iota_{\rm nor}(e_n)\in (\K_\X(M)^\sharp_J)^*$ be a weak$^*$ limit point and for any $T\in \X_0$, we have
$$e\iota_{\rm nor}(T)=\lim_n p_{\rm nor} \iota(e_n)\iota(T) p_{\rm nor} =\lim_n p_{\rm nor} \iota (e_nT) p_{\rm nor}=\iota_{\rm nor}(T),$$
and similarly $\iota_{\rm nor}(T) e= \iota_{\rm nor}(T)$.
By density of $\iota_{\rm nor}(\X_0)\subset (\K_\X(M)^\sharp_J)^*$, we conclude that $e$ is the identity in $(\K_\X(M)^\sharp_J)^*$.
\end{proof}

\begin{lem}\label{lem: projection commute}
Let $M$ be a finite von Neumann algebra and $B\subset M$ a von Neumann subalgebra.
Let $e_B\in \B(L^2M)$ be the orthogonal projection onto $L^2B$.
Then $\iota (e_B) \in \B(L^2M)^{**}$ commutes with $p_{\rm nor}$.
\end{lem}
\begin{proof}
Suppose $\B(L^2M)^{**}\subset \B(\cH)$ and notice that $\xi \in \cH$ is in the range of $p_{\rm nor}$ 
	if and only if $M\ni x\to \langle \iota(x)\xi ,\xi\rangle$ and $JMJ\ni x\to \langle \iota(x)\xi ,\xi\rangle$
	are normal.
For $\xi\in p_{\rm nor} \cH$, we have $\varphi(x):=\langle \iota(x) \iota(e_B)\xi, \iota(e_B) \xi\rangle=\langle \iota(E_B(x))\xi, \xi\rangle$
	is also normal for $x\in M$ and $JMJ$,
	which implies that $\iota(e_B) p_{\rm nor}=p_{\rm nor} \iota(e_B) p_{\rm nor}$.
It follows that $\iota(e_B)$ and $p_{\rm nor}$ commutes.
\end{proof}

\begin{lem}\label{lem: identity}
Let $\Gamma$ be a group and $\Lambda<\Gamma$ a subgroup.
Let $M=L\Gamma$, $B=L\Lambda$ and $\X= \X_B$.
Denote by $\{t_k\}_{k\in K}$ a representative of $\Gamma/ \Lambda$, i.e., $\Gamma=\sqcup _{k\in K} t_k \Lambda$
	and $u_k:=\lambda_{t_k}\in L\Gamma$ the canonical unitaries.
For each finite subset $F\subset K$, let $e_F=\bigvee_{k,\ell\in F} u_k Ju_\ell J e_B Ju_\ell ^*J u_k^*$.
Then $\lim_F \iota_{\rm nor}(e_F)\in ( \K_\X(M)_J^\sharp )^*$ is the identity.
\end{lem}
\begin{proof}
Denote by $\X_0\subset \B(L^2 M)$ the hereditary ${\rm C}^*$-subalgebra generated by $ x JyJ e_N$ for $x,y\in C^*_r(\Gamma)$.
It is clear that $\X_0$ is an $M$-boundary piece and by hereditariness we have $e_F\in\X_0$ for each $F$.

First we show that $\K_{\X_0}^{\infty,1}(M)=\K_\X^{\infty,1}(M)$, where $\K_{\X_0}^{\infty,1}(M)$ is obtained from $\X_0$ in the way described in Section \ref{sec:boundary piece}.
Notice that $\B(L^2 M) \X_0 \subset \K_\X^L(M)$ is dense in $\|\cdot\|_{\infty,2}$.
Indeed, for any contractions $T\in \B(L^2 M)$ and $x,y\in L\Gamma$, 
	we may find a net of contractions $T_i\in \B(L^2 M)\X_0$ such that $T_i\to Te_N x JyJ$
	in $\|\cdot\|_{\infty,2}$,
	as it follows directly from \cite[Proposition 3.1]{DKEP22}, the non-commutative Egorov theorem and the Kaplansky density theorem.
It then follows that $\K_{\X_0}(M)\subset \K_\X^{\infty,1}(M)$ is dense in $\|\cdot\|_{\infty,1}$
	and hence $\overline{\X_0}^{\infty,1}=\K_{\X_0}^{\infty,1}(M)=\K_\X^{\infty,1}(M)$ by \cite[Proposition 3.6]{DKEP22}.

Next we show that $\{e_F\}_F$ forms an approximate unit of $\X_0$.
Indeed, every element in $\X_0$ can be written as a norm limit of linear spans consisting of elements of the from $x_1 Jy_1 J T Jy_2 J x_2$,
	where $x_i,y_i\in C_r^*(\Gamma)$ and $T\in \B(L^2B)$.
Write each $x_i, y_i$ as summations of $u_k \lambda_t$, $t\in \Lambda$, it suffices to check $e_F (u_k J u_\ell J e_B)$ and $(e_B Ju_\ell J u_k) e_F$ agree with $u_k J u_\ell J e_B$ and $e_B Ju_\ell J u_k$ when $F$ is large enough, respectively, which follows easily from the construction of $e_F$.

By Lemma~\ref{lem: projection commute}, it is easy to check that $\iota(e_F)$ commutes with $p_{\rm nor}$ for every $F$.
And it follows from Lemma~\ref{lem: approx unit} that $\lim_F \iota_{\rm nor}(e_F)\in (\K_\X(M)^\sharp_J)^*$ is the identity.
\end{proof}

\begin{lem}\label{lem: malnormal orthogonal}
Let $\Gamma$ be a group and $\Lambda<\Gamma$ a subgroup.
Denote by $q_\K\in (\K(L^2M)^\sharp_J)^*$ the identity, $M=L\Gamma$ and $B=L\Lambda$.
Then $p_{t,s}=q_\K^\perp \iota_{\rm nor}(\lambda_t \rho_s e_B \rho_s^* \lambda_t^*)\in (\B(L^2M)^\sharp_J)^*$ is a projection for $t,s\in \Gamma$.
Moreover, if $\Lambda <\Gamma$ is almost malnormal, then $p_{t,s} p_{t', s'}=0$ if $t\neq t'$ or $s\neq s'$.
\end{lem}
\begin{proof}
Since $p_{\rm nor}$ commutes with $\iota(M)$ and $\iota(JMJ)$ 
	and $q_\K\in (\B(L^2M)^\sharp_J)^*$ is a central projection,
	together with  Lemma~\ref{lem: projection commute}, 
	we see that $p_{t,s}$ is a projection.

Note that $p_{t,s} p_{t', s'}=q_\K^\perp p_{\rm nor} \iota( {\rm Proj}_{t \Lambda s \cap t' \Lambda s'} ) p_{\rm nor}$, 	
	where $ {\rm Proj}_{t \Lambda s \cap t' \Lambda s'}$ denotes the orthogonal projection onto the ${\rm sp}\{t \Lambda s \cap t' \Lambda s'\}$,
	which is finite dimensional if $t\neq t'$ or $s\neq s'$ by the almost malnormality of $\Lambda<\Gamma$.
And it follows that $p_{t,s} p_{t', s'}=0$.
\end{proof}

\subsection{Proof of Proposition~\ref{prop:bootstrap}}
\begin{proof}
Since $N$ has no properly proximal direct summand, there exists an $N$-central state $\mu$ on $\tilde \bS(N)$ 
	such that $\mu_{|N}$ is normal and faithful by Lemma~\ref{lem: faithful state}.

Let $E:=\Ad(e_N):\B(L^2M)\to \B(L^2N)$ and 
	for the corresponding bidual $\tilde E: \big( \B(L^2 M)^\sharp_J \big)^*\to \big( \B(L^2 N)^\sharp_J \big)^*$,
	we have a u.c.p.\ map $\tilde E_{|\tilde \bS(M)}:\tilde \bS(M)\to \tilde \bS(N)$ 
	by Lemma~\ref{lem: bidual conditional expectation}.
Thus $\varphi=\mu\circ \tilde E_{|\tilde \bS(M)}: \tilde \bS(M)\to \C$ defines a $N$-central state that is faithful and normal on $M$.
Let $q_\K \in \big(\K(L^2M)^\sharp_J\big)^*$, 
	and $q_\X\in \big(\K_\X(M)^\sharp_J\big)^*$ be the corresponding identities in these von Neumann algebras.
Note that $q_\K\leq q_\X$ and $q_\X$ commutes with $M$ and $JMJ$.

First we analyze the support of $\varphi$.
Observe that $\varphi(q_{\K}^{\perp})=1$.
Indeed, if $\varphi(q_{\K})> 0$, i.e., $\varphi$ does not vanish on $(\mathbb K (L^2M)^\sharp_J)^*$, 
	then we may restrict $\varphi$ to $\mathbb B(L^2M)$, 
	which embeds into $(\mathbb K (L^2M)^\sharp_J)^*$ as a normal operator $M$-system \cite[Section 8]{DKEP22}, 
	and this shows that $N$ would have an amenable direct summand.
We also have $\varphi(q_{\X})=1$, since if $\varphi(q_{\X}^\perp)>0$, 
	we would then have an $N$-central state
$$\frac{1}{\varphi(q_{\X}^\perp)}\varphi\circ \Ad(q_{\X}^\perp): \tilde {\mathbb S}_{\mathbb X}(M)\to\mathbb C,$$ 
	whose restriction to $M$ is normal. 
This contradicts the assumption that $N\subset M$ is properly proximal relative to $\mathbb X$,
	since $\bS_\X(M)$ embeds unitally into $\widetilde{\bS}_\X(M)$ through $\iota_{\rm nor}$ in Section \ref{section: bidual boundary piece}.
Therefore we conclude that $\varphi(q_\X q_\K^\perp)=1$.

Let $B:= L\Lambda\subset M$ and $e_B: L^2 M\to L^2 B$ the orthogonal projection.

\noindent{\bf Claim.} There exists a u.c.p.\ map $\phi: \langle M, e_{B}\rangle \to q_\K^\perp q_\X \tilde \bS(M) q_\X$ 
	such that $\phi(x)=q_\K^\perp q_\X x$ for any $x\in M$.

This claim clearly implies that $N$ is amenable relative to $B$ inside $M$, 
	as $\nu=\varphi\circ\phi\in \langle M, e_{B}\rangle^*$ is an $N$-central state, 
	which is a normal faithful state when restricted to $M$.

\noindent{\bf Proof of  claim.}
Recall from Section~\ref{sec:bidual} that we may embed $\B(L^2M)$ into $(\B(L^2M)^\sharp_J)^*$ through the u.c.p.\ map $\iota_{\rm nor}$,
	which is given by $\iota_{\rm nor}=\Ad(p_{\rm nor})\circ \iota$,
	where $\iota:\B(L^2M)\to \B(L^2M)^{**}$ is the canonical $*$-homomorphism into the universal envelope,
	and $p_{\rm nor}$ is the projection in $\B(L^2M)^{**}$ such that $p_{\rm nor} \B(L^2M)^{**} p_{\rm nor} = (\B(L^2M)^\sharp_J)^*$.
We have that 
	$(\iota_{\rm nor})_{\mid M}$ and $(\iota_{\rm nor})_{\mid JMJ}$ are faithful normal representations of $M$ and $JMJ$, respectively,
	and to eliminate possible confusion, we will denote by $\iota_{\rm nor}(M)$ and $\iota_{\rm nor}(JMJ)$ the copies of $M$ and $JMJ$ in 
	$(\B(L^2M)^\sharp_J)^*$.
Restricting $\iota_{\rm nor}$ to ${\rm C}^*$-subalgebra $A\subset \B(L^2M)$ satisfying $M, JMJ\subset \M(A)$
	give rise to the embedding of $A$ into $(A^{\sharp}_J)^*$.
Furthermore, although $\iota_{\rm nor}$ is not a $*$-homomorphism, 
	by Lemma~\ref{lem: projection commute}, ${\rm sp}M e_B M$ is in the multiplicative domain of $\iota_{\rm nor}$.

Denote by $\{t_k\}_{k\geq 0}\subset \Gamma$ a representative of the cosets $\Gamma/\Lambda$ with $t_0$ being the identity of $\Gamma$,
	i.e., $\Gamma = \bigsqcup_{k\geq 0} t_k\Lambda$,
	and $u_k:=\lambda_{t_k}\in \cU(L\Gamma)$.
We will construct the map $\phi$ in the following steps.

{\bf Step 1.} For each $n\geq 0$, consider the u.c.p.\ map $\psi_n: \langle M, e_B\rangle \to \langle M, e_B\rangle$ given by 
	$\psi_n(x)=(\sum_{k\leq n} u_k e_B u_k^*) x (\sum_{\ell \leq n} u_\ell e_B u_\ell^*)$,
	and notice that $\psi_n$ maps $\langle M, e_B\rangle$ into the $*$-subalgebra $A_0 :={\rm sp}\{ u_k a e_B u_\ell ^*\mid a\in B, k,\ell\geq 0\}$.

{\bf Step 2.}
By Lemma~\ref{lem: malnormal orthogonal}, we have $\{\iota_{\rm nor}(Ju_k J e_B J u_k^* J)\}_{k\geq 0}\subset (\B(L^2M)^\sharp_J)^*$ 
	are pairwise orthogonal projections.
Set $e=\sum_{k\geq 0} \iota_{\rm nor}(Ju_k J e_B J u_k^* J) \in (\B(L^2M)^\sharp_J)^*$ 
	and notice that $e$ is independent of the choice of the representative $\Gamma/\Lambda$.
Put $\phi_0: A_0 \to  q_{\K}^\perp (\B(L^2M)^\sharp_J)^*$ to be $\phi_0(u_r a e_B u_\ell^*)= q_\K^\perp \iota_{\rm nor}(u_r a) e \iota_{\rm nor}(u_\ell^*)$

It is easy to see that $\phi_0$ is well-defined.
We then check that $\phi_0$ is a $*$-homomorphism.
For any $x\in M$, we claim that 
\begin{equation}\label{equa1}
q_\K^\perp e \iota_{\rm nor}(x) e = q_\K^\perp  \iota_{\rm nor}(E_B(x)) e.
\end{equation}
Indeed, 
\[
\begin{aligned}
&q_\K^\perp  e \iota_{\rm nor}(x) e \\
=&q_\K^\perp  \sum_{k,\ell \geq 0} \iota _{\rm nor}\big ((Ju_k J e_B Ju_k^* J) x( Ju_\ell J e_B Ju_\ell ^*J)\big ) \\
=&q_\K^\perp \iota_{\rm nor}(E_B(x))\sum_{k\geq 0} \iota_{\rm nor}( Ju_k J e_B Ju_k^* J) +\sum_{k\neq \ell} \iota_{\rm nor}\big ((Ju_k J e_B Ju_k^* J )x (Ju_\ell J e_B Ju_\ell ^*J )\big ).
\end{aligned}
\]
Since $\Lambda<\Gamma$ is almost malnormal which implies that $L^2(M\ominus B)$ is a mixing $B$-bimodule, one may check that 
$(Ju_k J e_B Ju_k^* J)( x-E_B(x))(Ju_\ell J e_B Ju_\ell ^*J)\in \B(L^2M)$ is a compact operator from $M$ to $L^2M$ if $\ell\neq k$.
We also have $(Ju_k J e_B Ju_k^* J)E_B(x)(Ju_\ell J e_B Ju_\ell ^*J)=0$ if $\ell\neq k$, 
	and it follows that 
$\sum_{k\neq \ell} q_\K^\perp  \iota_{\rm nor}(Ju_k J e_B Ju_k^* J x Ju_\ell J e_B Ju_\ell ^*J) =0$.

It then follows from (\ref{equa1}) that $\phi_0$ is a $*$-homomorphism.

We also show $\phi_0$ is norm continuous.
Set $\sum_{i=1}^d u_{k_i} a_i e_B u_{\ell_i}^*\in A_0$,
	and note that we may assume $k_i\neq k_j$ and $\ell_i\neq \ell_j$
	for $i\neq j$.
Consider $P_k=  q_\K^\perp \sum_{i=1}^d \iota_{\rm nor}({\rm Proj}_{t_{\ell_i} \Lambda t_{k}^{-1}})$
	and $Q_k = q_\K^\perp \sum_{i=1}^d \iota_{\rm nor}({\rm Proj}_{t_{k_i} \Lambda t_{k}^{-1}})$,
	where ${\rm Proj}_{t_{\ell_i} \Lambda t_{k}^{-1}}\in \B(\ell^2\Gamma)$ is the orthogonal projection onto the subspace $\overline{{\rm sp}\{\delta_t\mid t\in t_{\ell_i} \Lambda t_{k}^{-1}\}}^{\|\cdot\|}$,
	i.e., ${\rm Proj}_{t_{\ell_i} \Lambda t_{k}^{-1}}=J u_kJ u_{\ell_i} e_B u_{\ell_i}^* J u_k ^*J$.
By Lemma~\ref{lem: malnormal orthogonal}, 
	we have $P_k$ and $Q_k$ are a projections and $P_k P_r=Q_k Q_r=0$ if $k\neq r$.
Moreover, note that for each $i$, $\iota_{\rm nor}(e_B u_{\ell_i}^* Ju_k^*J) P_k= q_\K^\perp \iota_{\rm nor}(e_B u_{\ell_i}^* Ju_k^*J)$
	and $\iota_{\rm nor}(e_B u_{k_i}^* Ju_k^*J) Q_k=q_\K^\perp \iota_{\rm nor}(e_B u_{k_i}^* Ju_k^*J)$.
Let $\cH$ be the Hilbert space where $(\B(L^2M)^\sharp_J)^*$ is represented on.
For $\xi, \eta\in (\cH)_1$, we compute 
\[
\begin{aligned}
|\langle \phi_0(\sum_{i=1}^d u_{k_i} a_i e_B u_{\ell_i}^*)\xi, \eta\rangle|
&\leq \sum_{k\geq 0} | \sum_{i=1}^d \langle q_\K^\perp \iota_{\rm nor}(e_B u_{\ell_i}^* Ju_k ^* J)\xi, \iota_{\rm nor} ( Ju_k J u_{k_i}e_B a_i )^*\eta\rangle|\\
&= \sum_{k\geq 0} | \sum_{i=1}^d \langle  \iota_{\rm nor}(e_B u_{\ell_i}^* Ju_k ^* J)P_k\xi, \iota_{\rm nor} ( Ju_k J u_{k_i}e_B a_i )^*Q_k\eta\rangle|\\
&\leq \sum_{k\geq 0} \|\iota_{\rm nor}(Ju_k J(\sum_{i=1}^d  u_{k_i} a_i e_B u_{\ell_i}^*) Ju_k ^* J)\|\|P_k\xi \| \|Q_k \eta\|\\
&\leq \|\sum_{i=1}^d  u_{k_i} a_i e_B u_{\ell_i}^*\|(\sum_{k\geq 0} \|P_k\xi \|^2)^{1/2} (\sum_{k\geq 0} \|Q_k\eta \|^2)^{1/2}\\
&\leq \|\sum_{i=1}^d  u_{k_i} a_i e_B u_{\ell_i}^*\|,
\end{aligned}
\]
where the last inequality follows from the orthogonality of $\{P_k\}$ and $\{Q_k\}$.

Lastly, notice that $\phi_0$ maps $A_0$ into $q_\K^\perp \tilde \bS(M)$.
In fact, for any $s\in \Gamma$, we have
	$$ \iota_{\rm nor}(\rho_s)e\iota_{\rm nor}(\rho_s^*)=\sum_{k\geq 0}\iota_{\rm nor}(J(\lambda_s u_k)J e_B J (\lambda_su_k)^* J) = e,$$
	as $\bigsqcup_{k\geq 0} st_k \Lambda_j= \Gamma$,
	and it follows that $\phi_0(A_0)$ commutes with $\iota_{\rm nor}(JMJ)$.

Therefore, we conclude that $\phi_0$ is a norm continuous $*$-homomorphism from $A_0$ to $q_\K^\perp \tilde\bS(M)$
	and hence extends to the ${\rm C}^*$-algebra $A:=\overline{A_0}^{\|\cdot\|}$.

{\bf Step 3.} For each $n\geq 0$, set $\phi_n:= \phi_0\circ \psi_n: \langle M, e_B\rangle \to q_\K^\perp \tilde \bS(M)$,
	which is c.p.\ and subunital by construction.
We may then pick $\phi\in CB(\langle M, e_B\rangle, q_\K^\perp \tilde \bS(M))$ a weak$^*$ limit point of $\{\phi_n\}_n$, 
	which exists as $q_\K^\perp \tilde \bS(M)$ is a von Neumann algebra.

We claim that 
$$\Ad(q_\X)\circ \phi:\langle M, e_B\rangle \to q_{\K}^\perp q_\X \widetilde{\bS}(M) q_{\X} $$ 
is an $M$-bimodular u.c.p.\ map, which amounts to showing $\phi(x)=q_{\K}^\perp q_\X \iota_{\rm nor}(x)$ for any $x\in M$. 

In fact, for any $x\in M$, we have 
\[
\begin{aligned}
    \phi(x)=&\lim_{n\to\infty} \phi_0\Big(\sum_{0\leq k,\ell\leq n} (u_k  E_{B}( u_k^* x u_\ell)e_{B}  u_\ell ^*)\Big)\\
    =& q_{\K}^{\perp} \lim_{n\to\infty} \sum_{0\leq k,\ell\leq n} \iota_{\rm nor}( u_k  E_{B}( u_k^* x u_\ell) ) e\iota_{\rm nor}( u_\ell ^*)\\
	=&  q_{\K}^{\perp} \lim_{n\to\infty} \sum_{0\leq k,\ell\leq n} \big ( \iota_{\rm nor} (u_k) e \iota_{\rm nor}(u_k^*)\big ) 
	\iota_{\rm nor}( x)\big( \iota_{\rm nor}( u_\ell )e \iota_{\rm nor}(u_\ell^*)\big),\\
\end{aligned}
\]
where the last equation follows from (\ref{equa1}).
Finally, note that by Lemma~\ref{lem: malnormal orthogonal} $\{p_k\}_{k\geq 0}$ is a family of pairwise orthogonal projections, where
$$p_k:=q_\K^\perp  \iota_{\rm nor} (u_k) e \iota_{\rm nor}(u_k^*)=q_\K^\perp \sum_{r\geq 0} \iota_{\rm nor}(Ju_r J u_k e_B u_k^* Ju_r^*J),$$
	and $\sum_{k\geq 0}p_k= \sum_{k,r\geq 0}q_\K^\perp \iota_{\rm nor}(Ju_r J u_k e_B u_k^* Ju_r^*J)=q_\K^\perp q_\X$
	by Lemma~\ref{lem: identity}.
Therefore, we conclude that $\phi(x)=q_\K^\perp q_\X \iota_{\rm nor}(x)$, as desired.
\end{proof}

\section{From relative amenability to rigidity}\label{section: 3rd step}
In this section, we show that for
	von Neumann algebras arising from Gaussian actions,
	the associated s-malleable deformations converge uniformly 
	on subalgebras that are amenable relative to the acting group,
	provided that the orthogonal representations are 
	weakly contained in the left regular.

First we recall the construction of Gaussian actions and the associated s-malleable deformations \cite{Fu07, PeSi12}.
See e.g., \cite{KeLi16} for details on Gaussian actions.

Let $\mathcal H$ be a real Hilbert space, 
	the Gaussian process gives a tracial abelian von Neumann algebra $A_\cH$, 
	together with an isometry $S: \mathcal H \to L^2_{\mathbb R}(A_\cH)$ 
	so that orthogonal vectors are sent to independent Gaussian random variables, 
	and so that the spectral projections of vectors in the range of $S$ generate $A_\cH$ as a von Neumann algebra. 

In this case, the complexification of the isometry $S$ 
	extends to a unitary operator from the symmetric Fock space 
	$\mathfrak S(\mathcal H) = \mathbb C\Omega\oplus \bigoplus_{n=1}^\infty (\cH\otimes\mathbb C)^{\odot n}$ 
	into $L^2(A_\cH)$.  
If $\mathcal H = \mathcal H_1 \oplus \mathcal H_2$, 
	then conjugation by the unitary implementing the canonical isomorphism 
	$\mathfrak  S (\mathcal H_1 \oplus \mathcal H_2) \cong \mathfrak S(\mathcal H_1) \ovt \mathfrak S(\mathcal H_2)$ 
	implements a canonical isomorphism $A_{\mathcal H_1 \oplus \mathcal H_2} \cong A_{\mathcal H_1} \ovt A_{\mathcal H_2}$.

If $V: \mathcal K \to \mathcal H$ is an isometry, 
	then we obtain an isometry $V^{\mathfrak S} : \mathfrak S(\mathcal K) \to \mathfrak S(\mathcal H)$ 
	on the level of the symmetric Fock spaces, 
	and conjugation by this isometry gives an embedding of von Neumann algebras ${\rm Ad}(V^{\mathfrak S}): A_{\mathcal K} \to A_{\mathcal H}$.
If $V$ were a co-isometry the conjugation by $V^{\mathfrak S}$ implements 
	instead a conditional expectation from $A_{\mathcal H}$ to $A_{\mathcal K}$. 
In particular, if $U \in \mathcal O(\mathcal H)$ is an orthogonal operator, 
	then we obtain a trace-preserving $*$-isomorphism $\sigma_U = {\rm Ad}( U^{\mathfrak S} ) \in {\rm Aut}(A_{\mathcal H})$. 
If $\pi: \Gamma \to \mathcal O(\mathcal H)$ is an orthogonal representation, 
	then the Gaussian action associated to $\pi$, denoted by $\sigma_{\pi}$, is given by 
	$\Gamma \ni t \mapsto \sigma_{\pi(t)} \in {\rm Aut}(A_{\mathcal H})$. 
When $\pi$ is the left regular representation, the Gaussian action coincides with the Bernoulli action with diffuse base.

Now let $\pi: \Gamma\to \cO(\cH)$ be a fixed orthogonal representation of a countable group $\Gamma$ 
	and $\Gamma\actson^{\sigma_\pi} A_{\cH}$ the associated Gaussian action.
We recall the construction of the s-malleable deformation from \cite{PeSi12}.

Consider orthogonal matrices 
$$
V=
\begin{pmatrix}
1 & 0\\
0 & -1
\end{pmatrix}
\in \cO(\cH\oplus \cH)
\ {\rm and\ }
U_t=
\begin{pmatrix}
\cos(\pi t/2) & -\sin(\pi t/2)\\
\sin(\pi t/2) & \cos(\pi t/2)
\end{pmatrix}
\in \cO(\cH\oplus \cH),
$$
	for $t\in \R$. 
Let $\alpha_t=\sigma_{U_t}$ and $\beta=\sigma_{V}$ be the associated automorphisms of $A_{\cH \oplus \cH}\cong A_\cH\overline\otimes A_\cH$, 
	and both extend to $\Aut((A_\cH\overline\otimes A_\cH) \rtimes_{\sigma_\pi \otimes \sigma_\pi }\Gamma)$, still denoted by $\alpha_t$ and $\beta$,
	as $V$ and $U_t$ commute with $(\pi\oplus \pi)(\Gamma)$ and $\sigma_{\pi\oplus \pi}=\sigma_\pi \otimes \sigma_\pi$.
And $\alpha_t$ and $\beta$ form a s-malleable deformation in the sense of Popa \cite{Po06B} for $M:=A_{\cH}\rtimes_{\sigma_\pi}\Gamma$ inside $\tilde M:=( A_\cH\overline\otimes A_\cH) \rtimes_{\sigma_\pi \otimes \sigma_\pi }\Gamma$.

The following is an abstraction of \cite[Corollary 2.12]{Ioa15}.
Nevertheless we include the proof for completeness.

\begin{lem}\label{lem: abstract Ioana}
Let $(M, \tau)$ be a tracial von Neumann algebra, $B\subset M$ and $N\subset pMp$ von Neumann subalgebras with $p\in \cP(M)$.
Suppose there exist another tracial von Neumann algebra  $(\tilde M,\tilde \tau)$
	such that $M\subset \tilde M$ and $\tilde \tau_{\mid M}=\tau$,
	and a net of trace preserving automorphisms 
	$\{\theta_t\}_{t\in \R}\subset \Aut(\tilde M)$
	such that  ${\theta_t}_{\mid B}\in  \Aut (B)$,
	and such that ${\theta_t}_{| M}\to \id_M$ in the point-$\|\cdot\|_2$ topology, as $t\to 0$.
If $N$ is amenable relative to $B$ inside $M$,
	the for any $0<\delta \leq 1$, one of the following is true.
\begin{enumerate}
\item There exists $t_\delta>0$ such that $\inf_{u\in \cU(N)}\|E_M(\theta_{t_\delta}(u))\|_2>(1-\delta)\|p\|_2$.
\item There exists a net $\{\eta_k\}\subset \cK^\perp$, 
	where $\cK$ is the closure of $M e_B \tilde M$ inside $L^2(\langle \tilde M, e_B\rangle)$, 
	such that $\|x\eta_k- \eta_k x\|_2\to 0$ for all $x\in N$,
	$\limsup_k \|y \eta_k\|_2\leq 2\|y\|_2$ for all $y\in pMp$ and 
	$\limsup_k \|p \eta_k\|_2 >0$.
\end{enumerate}
\end{lem}

\begin{proof}
Since $N$ is amenable relative to $B$, there exists a net $\{\xi_n\}_{n\in I}\in L^2(p \langle M, e_B \rangle p)$
	such that $\|x \xi_n- \xi_n x\|_2\to 0$ for all $x\in N$ and $\langle y \xi_n, \xi_n\rangle$, $\langle \xi_n y, \xi_n\rangle\to \tau(y)$ for all $y\in p M p$ by \cite{OzPo10I}.
We may extend $\alpha_t$ to an automorphism on $\langle \tilde M, e_B \rangle$ as $\alpha_t$ leaves $B$ globally fixed.
Denote by $e$ the orthogonal projection from $L^2(\langle \tilde M, e_B\rangle)$ to $\cK$.

{\bf Claim.} For any $x\in N$, $y\in \tilde M$, $t\in \R$, we have
\begin{enumerate}
\item  $\lim_n \|y\alpha_t(\xi_n)\|_2^2 =\tau (y^* y \alpha_t(p)) \leq \|y\|_2^2$ and $\lim_n\|\alpha_t(\xi_n)y\|_2^2=\tau(yy^*\alpha_t(p))\leq \|y\|_2^2$. \label{item 1}
\item $\limsup_n \|y(e \alpha_t(\xi_n))\|_2\leq \|y\|_2.$ \label{item 2}
\item $\limsup_n\|x\alpha_t(\xi_n)-\alpha_t(\xi_n)x\|_2\leq 2\|\alpha_t(x)-x\|_2.$ \label{item 3}
\end{enumerate}

\begin{proof}[Proof of the claim.]
\renewcommand\qedsymbol{$\blacksquare$}
(\ref{item 1}) Note that since $\xi_n\in p\cK$, we have
\[
\|y\alpha_t(\xi_n)\|_2^2=\langle \alpha_{t}^{-1}(y^*y)\xi_n, \xi_n\rangle = \langle E_M(\alpha_t^{-1}(y^* y)) p \xi_n, p \xi_n\rangle
	\to \tau(p E_M(\alpha_t^{-1}(y^*y)) p )=\tau( y^* y \alpha_t(p)),
\]
and the second one follows similarly.

(\ref{item 2}) Observe that $(\tilde M\ominus M)\cK \perp \cK$, and hence
\[
\begin{aligned}
&\|ye \alpha_t(\xi_n)\|_2^2=\langle y^* y e\alpha_t(\xi_n), e \alpha_t(\xi_n)\rangle = \langle E_M(y^*y) e\alpha_t(\xi_n), e\alpha_t(\xi_n)\rangle\\
= &\langle e E_M(y^*y)^{1/2} \alpha_t(\xi_n), e E_M(y^*y)^{1/2} \alpha_t(\xi_n)\rangle
\leq \|E_M(y^*y)^{1/2} \alpha_t(\xi_n)\|_2^2,
\end{aligned}
\]
and $\|E_M(y^*y)^{1/2} \alpha_t(\xi_n)\|_2\leq \|E_M(y^*y)^{1/2}\|_2 = \|y\|_2$ by (\ref{item 1}).

(\ref{item 3}) Compute 
\[
\|x\alpha_t(\xi_n)- \alpha_t(\xi_n) x\|_2\leq \|(x-\alpha_t(x))\alpha_t(\xi_n)\|_2+\|\alpha_t(\xi_n)(x-\alpha_t(x))\|_2+\|x\xi_n- \xi_n x\|_2.
\]
\end{proof}

For each pair of $(t,n)$ with $t>0$ and $n\in \N$, let $\eta_{t,n}=\alpha_t(\xi_n)-e \alpha_t(\xi_n)$.
Fix a $0<\delta\leq 1$ and consider the following two cases.

{\bf Case 1.} There exists $t>0$ such that $\limsup_n\|\eta_{t,n}\|_2 < \delta\|p\|_2/2.$

{\bf Case 2.} For all $t>0$, $\limsup_n \|\eta_{t,n}\|_2\geq \delta\|p\|_2/2.$

In {\bf Case 1}, fix $x\in\cU(N)$ and compute
\[
\begin{aligned}
\|E_M(\alpha_t(x))\alpha_t(\xi_n)\|_2 &\geq \|e E_M(\alpha_t(x))\alpha_t(\xi_n)\|_2= \|e \alpha_t(x) e\alpha_t(\xi_n)\|_2\geq \|e \alpha_t(x) \alpha_t(\xi_n)\|_2-\|\eta_{t,n}\|_2\\
&\geq \|e (\alpha_t(\xi_n)\alpha_t(x))\|_2-\|x \xi_n- \xi_n x\|_2-\|\eta_{t,n}\|_2,
\end{aligned}
\] 
and 
\[
\|e(\alpha_t(\xi_n)\alpha_t(x))\|_2=\|e(\alpha_t(\xi_n))\alpha_t(x)\|_2\geq \|\alpha_t(\xi_n) \alpha_t(x)\|_2-\|\eta_{t,n}\|_2=\|\xi_n x\|_2-\|\eta_{t,n}\|_2.
\]
Altogether, we conclude that for any $x\in \cU(N)$,
\[
\begin{aligned}
\|E_M(\alpha_t(x))\|_2 &\geq \lim_n \|E_M(\alpha_t(x))\alpha_t(\xi_n)\|_2\geq \liminf_n( \|\xi_n x\|_2-2\|\eta_{t,n}\|_2-\|x\xi_n- \xi_n x\|_2)\\
&=\|p\|_2- 2\limsup_n \|\eta_{t,n}\|_2-\limsup_n \|x\xi_n-\xi_n x\|_2> (1-\delta)\|p\|_2.
\end{aligned}
\]

In {\bf Case 2}, let $k=(X, Y, \varepsilon)$ be a triple such that $X\subset N$, $Y\subset pMp$ are finite subsets and $\varepsilon>0$.
Then we may find $0<t_k\leq 1$ such that $\|x-\alpha_{t_k}(x)\|_2< \varepsilon/2$ for all $x\in X$ and $\|\alpha_{t_k}(p)-p\|_2 < (1-\sqrt{1-\delta^2/4}) \|p\|_2/2$.
Observe that for any $x\in X$
$$ \|x\eta_{t_k,n}- \eta_{t_k,n} x\|_2 = \|(1-e)(x \alpha_{t_k}(\xi_n)-\alpha_{t_k}(\xi_n) x)\|_2 \leq \|x \alpha_{t_k} (\xi_n)- \alpha_{t_k} (\xi_n) x\|_2,$$
and by (\ref{item 3})
$$\limsup_n\|x \alpha_{t_k} (\xi_n)- \alpha_{t_k} (\xi_n) x\|_2\leq 2\|x-\alpha_{t_k}(x)\|_2<\varepsilon.$$
For $y\in Y$, by (\ref{item 1}) and (\ref{item 2}) we have 
$$\|y \eta_{t_k,n}\|_2\leq \|y \alpha_{t_k}(\xi_n)\|_2+ \|y e \alpha_{t_k}(\xi_n)\|_2\leq 2\|y\|_2.$$

Furthermore, from (\ref{item 1}) we also have	
$$\limsup_n \|p \eta_{{t_k},n}\|\geq \limsup_n(\|p \alpha_{t_k}(\xi_n)\|_2-\|e \alpha_{t_k}(\xi_n)\|_2)
	= \|p \alpha_{t_k}(p)\|_2-\liminf_n\|e \alpha_{t_k}(\xi_n)\|_2,$$
and 
$$\liminf_n \|e\alpha_{t_k}(\xi_n)\|_2^2=\liminf_n (\|\alpha_{t_k}(\xi_n)\|_2^2-\|\eta_{{t_k},n}\|_2^2)=\lim_n \|\xi_n\|_2^2-\limsup_n\|\eta_{{t_k},n}\|_2^2\leq (1-\delta^2/4) \|p\|_2^2.$$
It follows that
\[
\begin{aligned}
\limsup_n \|p\eta_{{t_k},n}\|&\geq \|p\alpha_{t_k}(p)\|_2 - \sqrt{1-\delta^2/4} \|p\|_2 \\
&\geq  \|p\|_2-\|p-\alpha_{t_k}(p)\|_2- \sqrt{1-\delta^2/4} \|p\|_2\\
&>(1-\sqrt{1-\delta^2/4}) \|p\|_2/2.
\end{aligned}
\]

Altogether, we may find some $n\in I$ such that by putting $\eta_k=\eta_{t_k, n}$ we have 
\begin{enumerate}
\item $\|x\eta_k-\eta_k x\|_2\leq \varepsilon$ for all $x\in X$,
\item $\|y\eta_k\|_2\leq 2\|y\|_2+\varepsilon $ for all $y\in Y$,
\item $\|p\eta_k\|\geq (1-\sqrt{1-\delta^2/4}) \|p\|_2/2$.
\end{enumerate}
\end{proof}

\begin{prop}\label{prop:rel amen imply rigid}
Let $\Gamma$ be a nonamenable group and $\pi: \Gamma \to \cO(\cH)$ be a orthogonal representation
	such that $\pi \prec \lambda$.
Denote by $\Gamma\actson^{\sigma_\pi} A_\cH$ the associated Gaussian action and $M=A_\cH \rtimes_{\sigma_\pi} \Gamma$.
Suppose $N\subset pMp$ is a von Neumann subalgebra, for some $p\in \cP(M)$, with no amenable direct summand, such that $N$ is amenable relative to $L\Gamma$ inside $M$.
Then we have $\alpha_t\to \id_N$ uniformly on the unit ball of $N$ as $t\to 0$, 
	where $\alpha_t$ is the s-malleable deformation described above.
\end{prop}
\begin{proof}
Let $\tilde M =( A_\cH\overline\otimes A_\cH) \rtimes_{\sigma_\pi \otimes \sigma_\pi }\Gamma$ 
	and $\alpha_t$, $\beta\in \Aut(\tilde M)$ be as above.
Suppose there exists some $0<\delta\leq 1$ such that case (1) 
	of Lemma~\ref{lem: abstract Ioana} does not hold.
Then we have that there exists $\{\eta_k\}\in \cK^\perp$ as in the second case of Lemma~\ref{lem: abstract Ioana}.
Note that the $M$-$M$ bimodule $L^2(\langle \tilde M, e_{L\Gamma}\rangle)\ominus \cK$ is isomorphic to 
	$L^2(\tilde M\ominus M)\otimes_{L\Gamma} L^2 \tilde M$.
It is shown in \cite[Lemma 3.3]{Bo12} that $L^2(\tilde M\ominus M)$ is weakly contained in the coarse $M$-$M$ bimodule as $\pi\prec\lambda$,
and hence we have 
$$L^2(\langle \tilde M, e_{L\Gamma}\rangle)\ominus \cK\prec L^2 M\otimes (L^2M \otimes_{L\Gamma} L^2 \tilde M) \prec L^2M \otimes L^2M ,$$
as $M$-$M$ bimodules. 
It follows that there exists a u.c.p.\ map 
$$\phi: \B(L^2 M) \to \B(L^2(\langle \tilde M, e_{L\Gamma}\rangle)\ominus \cK)\cap (M^{\rm op})',$$
	such that $\phi_{\mid M}=\id_M$.
Therefore, we obtain a state $\varphi$ on $\B(L^2 M)$ given by 
	$$\varphi(\cdot)=\lim_k \|p \eta_k\|_2^{-2} \langle \phi(\cdot)p\eta_k, p\eta_k\rangle,$$
	which is $N$-central and restricts to a normal state on $pMp$.
This contradicts the assumption that $N$ has no amenable direct summands.

Therefore, we have that $\lim_{t\to 0} \inf_{u\in \cU(N)} \|E_M(\alpha_t(u))\|_2= \|p\|_2$.
It follows that $\sup_{u\in \cU(N)} \|\alpha_t(u)-E_M(\alpha_t(u))\|_2\to 0$
	as $t\to 0$ and hence
	$\alpha_t\to \id$ uniformly on $(N)_1$ by Popa's transversality inequality
	\cite[Lemma 2.1]{Po08}.
\end{proof}

\begin{cor}\label{cor:unitary conjugate}
Let $M$, $p\in \cP(M)$ and $N\subset pMp$ be as in Proposition~\ref{prop:rel amen imply rigid}.
Denote $Q=\cN_{pMp}(N)''$.
If $\pi$ is mixing, then $Q\prec_M L\Gamma$.
Moreover, if $\Gamma$ is an i.c.c.\ group, then there exists $u\in \cU(M)$ such that $u^* Q u\subset L\Gamma$.
\end{cor}
\begin{proof}
Since $A_\cH$ is abelian, $N$ is diffuse and $Q$ is type II$_1$, 
	the assertion $Q\prec_M L\Gamma$ follows directly from \cite[Theroem 3.4]{Bo12} and Proposition~\ref{prop:rel amen imply rigid}.
And the proof for the moreover part is contained in \cite[Proposition 2.3]{Bo13}.
\end{proof}

\section{Proofs of main theorems}
Now we are ready to prove Theorem~\ref{thm: absorption} and its corollaries.

\begin{proof}
First we may realize $M$ as $L(\Z \wr \Gamma)$
	and note that $\Z\wr\Gamma$ is biexact relative to $\Gamma$ \cite[Corollary 15.3.9]{BrOz08}.
By Proposition~\ref{prop:dichotomy} we have that $N\subset M$ is properly proximal relative to $\X_{L\Gamma}$ 
	as $N$ has no amenable direct summand,
	where $\X_{L\Gamma}$ is the $M$-boundary piece associated with $L\Gamma$.
Moreover, since $\Gamma < \Z \wr \Gamma$ is almost malnormal and $N$ has no properly proximal direct summand, 
	we have $N$ is amenable relative to $L\Gamma$ inside $M$
	by Proposition~\ref{prop:bootstrap}.
The rest follows from Proposition~\ref{prop:rel amen imply rigid} by setting $\pi=\lambda$.
\end{proof}

\begin{proof}[Proof of Theorem~\ref{thm: cocycle superrigid}]
Let $\sigma: \Gamma\actson X$ be the Bernoulli action,
	$M=L^\infty(X)\rtimes_\sigma \Gamma$.
Set $\tilde M= (L^\infty(X)\overline \otimes L^\infty(X))\rtimes_{\tilde \sigma}\Gamma$, where $\tilde \sigma=\sigma\otimes \sigma$.
If we denote by $\sigma_t\in \cU(L^2(X))$ for each $t\in \Gamma$ the unitary that implements the action $\sigma$,
	then we have $M\subset \tilde M$ is generated by canonical unitaries $\{u_t=\tilde \sigma_t\otimes \lambda_t\mid t\in \Gamma\}$ and $L^\infty(X)\otimes \C$, where $\tilde \sigma_t=\sigma_t\otimes\sigma_t$.

Let $\Gamma_0< \Gamma$ be a nonamenable wq-normal subgroup that is not properly proximal.	
If $\omega: \Gamma\times X\to \T$ is a $1$-cocycle associated with $\sigma$,
	then for each $t\in \Gamma$, we may consider $\omega_t\in \cU(L^\infty(X))$ given by $\omega_t(x)=\omega(t, t^{-1}x)$
	and $\widetilde {L(\Gamma_0)} :=\{\tilde u_t:=\omega_t u_t\mid t\in \Gamma_0\}''\subset M$,
	which is a von Neumann subalgebra isomorphic to $L(\Gamma_0)$.

Since $\widetilde {L(\Gamma_0)}\cong L(\Gamma_0)$ has no amenable and no properly proximal direct summand by Remark~\ref{rem: no direct summand}, 
	it follows from Theorem~\ref{thm: absorption} that $\alpha_t$ converges to identity uniformly on the unit ball of $\widetilde {L(\Gamma_0)}$.
The result follows from \cite{Po07}.
\end{proof}

\begin{proof}[Proof of Theorem~\ref{thm: betti number}]
This is an immediate result of Theorem~\ref{thm: cocycle superrigid} and \cite[Theorem 1.1]{PeSi12}.
\end{proof}

\begin{proof}[Proof of Theorem~\ref{cor: conjugate}]
Since $\Lambda$ is exact, we have $L^\infty(Y)\rtimes_r\Lambda$ is an exact ${\rm C}^*$-algebra (e.g. \cite[Theorem 10.2.9]{BrOz08})
	and it follows that $(L^\infty(X)\rtimes\Gamma)^t\cong L^\infty(Y)\rtimes \Lambda$ is a weakly exact von Neumann algebra \cite{Kir95a}.
Since weak exactness is stable under amplifications and passes to von Neumann subalgebras (with normal conditional expectations) \cite[Corollary 14.1.5]{BrOz08}, 
	we have $L\Gamma$ is weakly exact, which implies $\Gamma$ is exact \cite{Oza07}.

Let $M=L^\infty(X)\rtimes\Gamma$, $N=L^\infty(Y)\rtimes \Lambda$ 
	and $\Lambda_0\lhd \Lambda$ be the nonamenable normal subgroup that is not properly proximal.
Since $N\cong M^t$, we may denote by $\theta: N^{1/t}\to M$ a $*$-isomorphism,
	and identify $N^{1/t}$ with $p \M_n(N) p$, where $n=\lceil 1/t \rceil$, $p=\diag(1,\dots, 1, p_0)\in \M_n(N)$ 
	and $p_0\in L(\Lambda_0)$ with $\tau_N(p_0)=1/t-\lfloor 1/t\rfloor$.

Note that by Remark~\ref{rem: no direct summand}, $\theta(p\M_n(L(\Lambda_0))p)\subset M$ satisfies the assumption of Theorem~\ref{thm: absorption}, 
	and hence by Corollary \ref{cor:unitary conjugate} we may find some $u\in \cU(M)$ such that $\alpha (p\M_n(L\Lambda)p)\subset L\Gamma$,
	where $\alpha:=\Ad(u)\circ \theta $.
Set $e=\diag(1, 0, \dots ,0)\in \M_n(L\Lambda)$ and we have $\alpha(L\Lambda)=\alpha(e \M_n(L\Lambda) e)\subset q L\Gamma q$,
	where $q=\alpha(e)\in L\Gamma$ and $\tau_M(q)=\tau_{N^{1/t}}(e)=t$.

It then follows from Popa's conjugacy criterion for Bernoulli actions \cite[Theorem 0.7]{Po06C} 
	(see also \cite[Theorem 6.3]{Ioa11}) that $t=1$ and there exist a unitary $v\in M$, 
	a character $\eta\in \Lambda$ and a group isomorphism $\delta: \Lambda\to \Gamma$ 
	such that $\alpha(L^\infty(Y))=v L^\infty (X) v^*$ and $\alpha(\lambda_t)=\eta(t) v \lambda_{\delta(t)} v^*$ 
	for any $t\in \Lambda$.
\end{proof}

\begin{proof}[Proof of Theorem~\ref{cor: OE superrigidity}]

A direct consequence of Theorem~\ref{cor: conjugate}, \cite{IsPeRu19} and \cite{Oza07}.
\end{proof}

\bibliographystyle{amsalpha}
\bibliography{ref}

\end{document}